\documentclass[amscd,amssymb,verbatim]{amsart}
\usepackage{amssymb,amsfonts,amsmath,amscd}

\theoremstyle{plain}
\newtheorem{Thm}{Theorem}[section]

\newtheorem{Prop}[Thm]{Proposition}

\theoremstyle{definition}

\theoremstyle{remark}

\newcommand{\cS}{{\mathcal S}}
\newcommand{\cT}{{\mathcal T}}

\begin{document}
\title[A Bump in the Road]{A Bump in the Road in Elementary Topology}
\author{Bruce Blackadar}
\address{Department of Mathematics and Statistics/0084 \\ University of Nevada, Reno \\ Reno, NV 89557, USA}
\email{bruceb@unr.edu}

\date{January 24, 2017}

\maketitle
\begin{abstract}
We observe a subtle and apparently generally unnoticed difficulty with the definition of the relative topology
on a subset of a topological space, and with the weak topology defined by a function.
\end{abstract}

\begin{quote}
`` `Obvious' is the most dangerous word in mathematics.''

\flushright
{\em E.\ T.\ Bell}\footnote{\cite[p.\ 16]{BellMathematics}.}
\end{quote}

\section{Relative Topology}

One of the most elementary constructions in general topology is the definition of the relative or
subspace topology on a subset of a topological space.  But it turns out it is not quite as elementary to
do this properly as has generally been thought.

If $(X,\cT)$ is a topological space and $Y\subseteq X$, the {\em relative topology},
or {\em subspace topology}, on $Y$ from $\cT$ is
$$\cT_Y=\{U\cap Y:U\in\cT\}$$
i.e.\ the open sets in $Y$ (called the {\em relatively open sets}) are the intersections with $Y$ of the open sets in $X$. 

The main issue we discuss is whether $\cT_Y$ is really a topology on $Y$.  This is generally
considered ``obvious'' or ``trivial.''  We write out the ``obvious'' argument: 

\paragraph{}
\begin{Prop}\label{RelTopProp}
$\cT_Y$ is a topology on $Y$.
\end{Prop}

\begin{proof}
We have $\emptyset=\emptyset\cap Y$ and $Y=X\cap Y$, so $\emptyset,Y\in\cT_Y$.
If $U_1\cap Y,\dots,U_n\cap Y\in\cT_Y$, where $U_1,\dots,U_n\in\cT$, then
$$(U_1\cap Y)\cap\cdots\cap(U_n\cap Y)=(U_1\cap\cdots\cap U_n)\cap Y\in\cT_Y$$
since $U_1\cap\cdots\cap U_n\in\cT$.
If $\{U_i\cap Y:i\in I\}$ is a collection of sets in $\cT_Y$, where each $U_i\in\cT$, then 
$$\bigcup_{i\in I}(U_i\cap Y)=\left ( \bigcup_{i\in I}U_i\right ) \cap Y\in\cT_Y$$
since $\cup_{i\in I}U_i\in\cT$.
\end{proof}

Most standard topology references, e.g.\ \cite{Bourbaki}, \cite{Dugundji}, \cite{Engelking}, \cite{HallS}, \cite{HockingY},
\cite{Kasriel}, \cite{Kelley}, \cite{Munkres}, \cite{Willard}, either give this argument explicitly or state that
the result is ``trivial'' or ``easily verified,'' presumably using this argument.

But actually there is a subtle problem with the last part of the argument:  how do we know that
every indexed collection of sets in $\cT_Y$ is of the form  $\{U_i\cap Y:i\in I\}$ for some $U_i\in\cT$?
In fact, the Axiom of Choice (AC) is needed to assert this, since for a given $V\in\cT_Y$ there
are in general many $U\in\cT$ for which $V=U\cap Y$, and one must somehow be chosen.  
(The same comment might apply to the finite intersection argument, but there
only finitely many choices need to be made so the AC is not needed.)  

When the AC was first formulated and its nature understood, it was observed that mathematicians
had already been using it extensively without comment and generally without notice.  The relative topology
example shows that this is still happening.

But does \ref{RelTopProp} really require the AC?  There is, in fact, a simple way to avoid it: 
there is a systematic way to choose the $U_i$ (I am indebted to
S.\ Jabuka for this observation).  If $V\in\cT_Y$, there is a largest open set $U\in\cT$  such that
$V=U\cap Y$, namely the union of all $W\in\cT$ for which $V=W\cap Y$.  A correct phrasing of
the proof would thus be:

\begin{proof}
We have $\emptyset=\emptyset\cap Y$ and $Y=X\cap Y$, so $\emptyset,Y\in\cT_Y$.
If $V_1,\dots,V_n\in\cT_Y$, then, for each $k$, $V_k=U_k\cap Y$ for some
$U_k\in\cT$; so 
$$V_1\cap\cdots\cap V_n=(U_1\cap Y)\cap\cdots\cap(U_n\cap Y)=(U_1\cap\cdots\cap U_n)\cap Y\in\cT_Y$$
since $U_1\cap\cdots\cap U_n\in\cT$.
If $\{V_i:i\in I\}$ is a collection of sets in $\cT_Y$, for each $i\in I$ let $U_i$ be the union of all $W\in\cT$ such that
$V_i=W\cap Y$.  Then $U_i\in\cT$ and $V_i=U_i\cap Y$ for each $i$, so
$$\bigcup_{i\in I}V_i=\bigcup_{i\in I}(U_i\cap Y)=\left ( \bigcup_{i\in I}U_i\right ) \cap Y\in\cT_Y$$
since $\cup_{i\in I}U_i\in\cT$.
\end{proof}

There is an alternate argument which avoids 
the AC in
\cite{Kuratowski} (which is the only topology book I have found with a complete correct proof of 
\ref{RelTopProp}).  Recall that a {\em Kuratowski closure operation} on a set $Y$ is an assignment
$A\mapsto\bar A$ for each subset $A$ of $Y$, with the properties $\bar\emptyset=\emptyset$, 
$A\subseteq\bar A=\bar{\bar A}$ for all $A$, and $\overline{A\cup B}=\bar A\cup\bar B$ for all $A,B$.
It is easy to show (the argument is in many standard references and can be found in
\cite{Blackadar}, and does not use the AC) that any 
Kuratowski closure operation defines closure with respect to a unique
topology for which the closed sets are precisely the sets $A$ for which $\bar A=A$.

To prove \ref{RelTopProp}, for $A\subseteq Y$ define $\tilde A=\bar A \cap Y$, where $\bar A$
is the closure of $A$ in $X$.  It is nearly trivial to check (without using the AC) that $A\mapsto
\tilde A$ is a Kuratowski closure operation on $Y$, and that the closed sets with respect to the
corresponding topology are precisely the complements of the sets in $\cT_Y$.  It follows that
the topology defined by this closure operation is $\cT_Y$, and in particular $\cT_Y$ is a topology.

\section{The Weak Topology Defined by a Function}

If $(X,\cT)$ is a topological space, $Y$ a set, and $f:Y\to X$ a function, there is a weakest topology
on $Y$ making $f$ continuous.  It should be
$$\cT_Y=\{f^{-1}(U):U\in\cT\}\ .$$
But is $\cT_Y$ actually a topology?  If $f$ is surjective, there is no difficulty verifying this (using that
preimages respect unions and intersections).  However, if $f$ is not surjective, we run into the same
problem as in \ref{RelTopProp} (which is actually just the case where $f$ is injective), since many
different open sets in $X$ can have the same preimage in $Y$, so the AC must apparently be used
to show that $\cT_Y$ is a topology.  

To show that $\cT_Y$ is a topology without using the AC,  the union trick works, and
the argument via Kuratowski closure operations  
works here too: for $A\subseteq Y$, set $\tilde A=f^{-1}(\overline{f(A)})$.
There is also an alternate argument.  
Let $Z=f(Y)\subseteq X$.  By \ref{RelTopProp} we have that
$\cT_Z$ is a topology on $Z$.  If $\cS$ is a topology on $Y$, then
$f$ is continuous as a function from $(Y,\cS)$ to $(X,\cT)$ if and only if it is continuous as a
function from $(Y,\cS)$ to $(Z,\cT_Z)$, and it is easily verified that $\cT_Y=(\cT_Z)_Y$.  But $f:Y\to Z$ is surjective,
so the AC is not needed to prove that $(\cT_Z)_Y$ is a topology.

\section{Should We Worry About the AC?}

Most modern mathematicians have no serious qualms about using the AC, and largely share
the opinion of Ralph Boas \cite[p.\ xi]{BoasPrimer}:
\begin{quote}
``[A]fter G\"{o}del's results, the assumption of the axiom of choice can do no mathematical harm
that has not already been done.''
\end{quote}

\bigskip

\noindent
As an analyst, I have no qualms about it myself.  But I do believe:
\begin{enumerate}
\item[1.]  When the AC is used, it should be mentioned.
\item[2.]  The AC should not be used if it is not needed.
\end{enumerate}
There is a gray area with 2: the AC can drastically simplify proofs of some results which can be
proved without it.  But the relative topology case is one where use of the AC is of 
highly doubtful benefit.

\bibliography{reltopref}

\begin{thebibliography}{Mun75}

\bibitem[Bel87]{BellMathematics}
E.~T. Bell.
\newblock {\em Mathematics, Queen and servant of science}.
\newblock MAA Spectrum. Mathematical Association of America, Washington, DC,
  1987.
\newblock Reprint of the 1951 original, With a foreword by Martin Gardner.

\bibitem[Bla]{Blackadar}
Bruce Blackadar.
\newblock {\em Real Analysis}.
\newblock http://wolfweb.unr.edu/homepage/bruceb/Meas.pdf.

\bibitem[Boa96]{BoasPrimer}
Ralph~P. Boas.
\newblock {\em A primer of real functions}, volume~13 of {\em Carus
  Mathematical Monographs}.
\newblock Mathematical Association of America, Washington, DC, fourth edition,
  1996.
\newblock Revised and with a preface by Harold P. Boas.

\bibitem[Bou98]{Bourbaki}
Nicolas Bourbaki.
\newblock {\em General topology. {C}hapters 1--4}.
\newblock Elements of Mathematics (Berlin). Springer-Verlag, Berlin, 1998.
\newblock Translated from the French, Reprint of the 1989 English translation.

\bibitem[Dug78]{Dugundji}
James Dugundji.
\newblock {\em Topology}.
\newblock Allyn and Bacon Inc., Boston, Mass., 1978.
\newblock Reprinting of the 1966 original, Allyn and Bacon Series in Advanced
  Mathematics.

\bibitem[Eng89]{Engelking}
Ryszard Engelking.
\newblock {\em General topology}, volume~6 of {\em Sigma Series in Pure
  Mathematics}.
\newblock Heldermann Verlag, Berlin, second edition, 1989.
\newblock Translated from the Polish by the author.

\bibitem[HS55]{HallS}
Dick~Wick Hall and Guilford~L. Spencer, II.
\newblock {\em Elementary topology}.
\newblock John Wiley \& Sons, Inc., New York; Chapman \& Hall, Ltd., London,
  1955.

\bibitem[HY88]{HockingY}
John~G. Hocking and Gail~S. Young.
\newblock {\em Topology}.
\newblock Dover Publications Inc., New York, second edition, 1988.

\bibitem[Kas09]{Kasriel}
Robert~H. Kasriel.
\newblock {\em Undergraduate topology}.
\newblock Dover Publications, Inc., Mineola, NY, 2009.
\newblock Reprint of the 1971 original [MR0283741].

\bibitem[Kel75]{Kelley}
John~L. Kelley.
\newblock {\em General topology}.
\newblock Springer-Verlag, New York, 1975.
\newblock Reprint of the 1955 edition [Van Nostrand, Toronto, Ont.], Graduate
  Texts in Mathematics, No. 27.

\bibitem[Kur66]{Kuratowski}
K.~Kuratowski.
\newblock {\em Topology. {V}ol. {I}}.
\newblock New edition, revised and augmented. Translated from the French by J.
  Jaworowski. Academic Press, New York, 1966.

\bibitem[Mun75]{Munkres}
James~R. Munkres.
\newblock {\em Topology: a first course}.
\newblock Prentice-Hall Inc., Englewood Cliffs, N.J., 1975.

\bibitem[Wil04]{Willard}
Stephen Willard.
\newblock {\em General topology}.
\newblock Dover Publications, Inc., Mineola, NY, 2004.
\newblock Reprint of the 1970 original [Addison-Wesley, Reading, MA;
  MR0264581].

\end{thebibliography}
\bibliographystyle{alpha}

\end{document}